\documentclass{amsart}
\usepackage{amsmath}
\usepackage{amssymb} 
\usepackage{amscd} 
\usepackage{graphicx} 
\usepackage{epsfig}
\usepackage[dvips]{color}
\usepackage{marvosym}
\usepackage{enumerate}
\usepackage{fullpage}
\usepackage[bindingoffset=0.2in,left=1in,right=1in,top=1in,bottom=1in,footskip=.25in]{geometry}

\newtheorem{theorem}{Theorem}[section]
\newtheorem{lemma}[theorem]{Lemma}

\newtheorem{proposition}[theorem]{Proposition}
\newtheorem{corollary}[theorem]{Corollary}

\newtheorem{question}[theorem]{Question}

\theoremstyle{definition}

\newtheorem{example}[theorem]{Example}

\newtheorem{remark}[theorem]{Remark}

\numberwithin{equation}{section}
\numberwithin{figure}{section}
\numberwithin{table}{section}

\newcommand{\comment}[1]{}

\newcommand{\bdry}{\ensuremath{\partial}}

\newcommand{\nbhd}{\ensuremath{\mathcal{N}}}

\renewcommand{\)}{\textup{)}}

\begin{document}
\baselineskip 14pt

\title{Tight fibered knots and band sums}

\author{Kenneth L. Baker and Kimihiko Motegi}

\address{Department of Mathematics, University of Miami, 
Coral Gables, FL 33146, USA}
\email{k.baker@math.miami.edu}
\address{Department of Mathematics, Nihon University, 
3-25-40 Sakurajosui, Setagaya-ku, 
Tokyo 156--8550, Japan}
\email{motegi@math.chs.nihon-u.ac.jp}

\thanks{The first named author was partially supported by a grant from the Simons Foundation (\#209184 to Kenneth L.\ Baker).  
The second named author has been partially supported by Japan Society for the Promotion of Science, 
Grants--in--Aid for Scientific Research (C), 26400099 and Joint Research Grant of Institute of Natural Sciences at Nihon University for 2015.}

\dedicatory{}

\begin{abstract}
We give a short proof that if a non-trivial band sum of two knots results in a tight fibered knot, 
then the band sum is a connected sum. 
In particular, this means that any prime knot obtained by a non-trivial band sum is not tight fibered. 
Since a positive L-space knot is tight fibered, 
a non-trivial band sum never yields an L-space knot.  
Consequently, any knot obtained by a non-trivial band sum cannot admit a finite surgery.

For context, 
we exhibit two examples of non-trivial band sums of tight fibered knots producing prime knots: one is fibered but not tight, 
and the other is strongly quasipositive but not fibered. 
\end{abstract}

\maketitle

{
\renewcommand{\thefootnote}{}
\footnotetext{2010 \textit{Mathematics Subject Classification.}
Primary 57M25, 57M27
\footnotetext{ \textit{Key words and phrases.}
tight fibered knot, strongly quasipositive knot, L-space knot, band sum}
}
}

\section{Introduction}
\label{section:Introduction}

Let $K_1 \sqcup K_2$ be a $2$-component split link in the $3$-sphere $S^3$. 
Let $\beta\colon [0, 1] \times [0, 1] \to S^3$ be an embedding such that 
$\beta([0, 1] \times [0, 1]) \cap K_1 =  \beta([0, 1] \times \{ 0\})$ and 
$\beta([0, 1] \times [0, 1]) \cap K_2 =  \beta([0, 1] \times \{ 1\})$. 
Then we obtain a knot $K_1 \natural_{\beta} K_2$ by replacing $\beta([0, 1] \times \{0, 1\})$ in $K_1 \cup K_2$ 
with $\beta(\{0, 1\} \times [0, 1])$. 
We call $K_1 \natural_{\beta} K_2$ a \textit{band sum} of $K_1$ and $K_2$ with the band $\beta$. 
In the following, 
for simplicity,  
we use the same symbol $\beta$ to denote the image 
$\beta([0, 1] \times [0, 1])$. 
We say that a band sum is \textit{trivial} if one of $K_1$ and $K_2$, say $K_2$, 
is the unknot and the band $\beta$ gives just a connected sum.  
If the band sum is trivial, then obviously $K_1 \natural_\beta K_2 = K_1$. 
The converse also holds, i.e.\ if $K_1 \natural_\beta K_2 = K_1$, 
then the band sum is trivial; see \cite{gabaiband, Sch}. 
A band sum is regarded as a natural generalization of the connected sum, 
and many prime knots are obtained by band sums; see \cite{kobayashi2}.

We say a fibered knot in $S^3$ is {\em tight} if, 
as an open book for $S^3$, it supports the positive tight contact structure on $S^3$.   
A knot in $S^3$ is {\em strongly quasipositive} if it is the boundary of a {\em quasipositive} Seifert surface, 
a special kind of Seifert surface obtained from parallel disks by attaching positive bands in a particular kind of braided manner, 
for a more precise definition see 
e.g.\ \cite[61.Definition]{Rudolph_HandBook}. 
It is shown by Rudolph \cite[Characterization Theorem]{Rudolph_Topology} (cf.\ \cite[90.Theorem]{Rudolph_HandBook}) 
that a knot $K$ is strongly quasipositive if and only if it has a Seifert surface that is a subsurface of the fiber of some positive torus knot. 
Hedden proved that tight fibered knots are precisely the fibered strongly quasipositive knots \cite[Proposition~2.1]{hedden}; 
Baader-Ishikawa \cite[Theorem~3.1]{BI} provides an alternative proof.  
From this correspondence, 
one may observe that a connected sum of two tight fibered knots is again a tight fibered knot.
(Of course this follows more directly from the definitions of connected sums of contact manifolds and contact structures supported by open books, e.g.\ \cite{etnyre, geiges}.)

The aim of this note is to prove: 

\begin{theorem}
\label{thm:main}
If a tight fibered knot in $S^3$ is a non-trivial band sum, 
then the band sum expresses the knot as a connected sum. 
In particular, it is not prime. 
\end{theorem}

\begin{corollary}
\label{cor:contrapos}
If a prime knot in $S^3$ is a non-trivial band sum, then it is not a tight fibered knot.
\end{corollary}

\medskip

A knot $K$ in the $3$--sphere $S^3$ is called an \textit{L-space knot} if 
it admits a nontrivial Dehn surgery yielding an L-space, 
a rational homology sphere whose Heegaard Floer homology is as simple as possible \cite{OSlens}.  
It is a {\em positive} (resp.\ {\em negative}) L-space knot if a positive (resp.\ negative) Dehn surgery yields an L-space. 
In \cite{krcatovicharxiv} Krcatovich proves that if a knot is a connected sum of non-trivial knots, 
then it is not an L-space knot.   
Since positive L-space knots are special types of tight fibered knots \cite{hedden, ni}, 
Theorem~\ref{thm:main} allows us to generalize Krcatovich's result to non-trivial band sums:

\begin{corollary}
\label{cor:main}
If a knot in $S^3$ is a non-trivial band sum, then it is not an L-space knot.
\end{corollary}

Since lens spaces, and more generally $3$--manifolds with finite fundamental group, are L-spaces \cite{OSlens}, 
Corollary~\ref{cor:main} immediately implies: 

\begin{corollary}
\label{finite_surgery}
Any knot obtained by a non-trivial band sum does not admit a finite surgery. 
\end{corollary}

Recall from \cite{hedden, BI} that the set of strongly quasipositive, fibered knots coincides with that of tight fibered knots. 
Thus Theorem~\ref{thm:main} says that any prime knot obtained by a non-trivial band sum fails to satisfy at least one of 
conditions (i) $K$ is strongly quasipositive, and (ii) $K$ is fibered. 
   
In fact, we demonstrate: 

\begin{proposition}
\label{examples}
\begin{enumerate}
\item
There exist tight fibered knots $K_1$ and $K_2$, 
together with a band $\beta$, such that 
$K_1 \natural_{\beta} K_2$ is  prime, strongly quasipositive, but not fibered. 

\item
There exist tight fibered knots $K_1$ and $K_2$, 
together with a band $\beta$, such that 
$K_1 \natural_{\beta} K_2$ is prime, fibered, but not tight \(hence not strongly quasipositive\). 
\end{enumerate}
\end{proposition}

The connected sum of two tight fibered knots is tight fibered. 
On the other hand we show:

\begin{proposition}
\label{nonSQP}
For any tight fibered knots $K_1$ and $K_2$, 
there exists a band $\beta$ such that $K_1 \natural_{\beta} K_2$ is fibered, but not tight \(hence not strongly quasipositive\). 
\end{proposition}

\section{Proofs of Theorem~\ref{thm:main} and Corollary~\ref{cor:main}}

\begin{proof}[Proof of Theorem~\ref{thm:main}.]
Assume a knot $K=K_1 \natural_\beta K_2$ in $S^3$ is a non-trivial band sum  of the split link $K_1 \sqcup K_2$ of knots $K_1$ and $K_2$ with a band $\beta$.  
Then $g(K) \geq g(K_1)+g(K_2)$ since genus is superadditive for band sums by Gabai \cite[Theorem~1]{gabaiband} and Scharlemann \cite[8.4 Theorem]{Sch}.  
Since  $K$ is concordant to $K_1 \# K_2$ by Miyazaki \cite[Theorem~1.1]{miyazakiband}, 
$\tau(K) = \tau(K_1 \# K_2)$ where $\tau$ is the Ozv\'ath-Sz\'abo concordance invariant \cite{OS}. 
Furthermore, additivity of $\tau$ under connected sum \cite[Proposition~3.2]{OS} shows 
$\tau(K_1 \# K_2) = \tau(K_1) + \tau(K_2)$.  
If we also suppose that $\tau(K) = g(K)$,
then because $\tau$ gives a lower bound on genus for all knots in $S^3$ \cite[Corollary~1.3]{OS}, 
we will have the string of inequalities: 
\[\tau(K) = g(K) \geq g(K_1)+g(K_2) \geq  \tau(K_1) + \tau(K_2) = \tau(K).\]
It then follows that $g(K)=g(K_1)+g(K_2)$.  

\medskip

Now for the proposition, assume the band sum $K$ is a tight fibered knot.  
It follows from Hedden \cite[Proposition~2.1]{hedden} and also   
Baader-Ishikawa \cite[Theorem~3.1]{BI} that $K$ is strongly quasipositive. 
Hence Livingston \cite[Theorem~4]{Liv} shows that $\tau(K) = g(K)$.  
It then follows from the above calculation that we have $g(K)=g(K_1)+g(K_2)$.

Now by Kobayashi \cite[Theorem~2]{kobayashi}, 
both $K_1$ and $K_2$ are fibered,  
and the banding expresses $K$ as the connected sum of $K_1$ and $K_2$, 
i.e.\ there is a $2$--sphere which split $K_1$ and $K_2$ and intersects $\beta$ in a single arc. 
The proof of \cite[Proposition 4.1]{kobayashi} clarifies this last point. 
In particular, $K$ cannot be prime. 
For otherwise, 
$K_1$ or $K_2$ would be a trivial knot and the band sum is trivial, a contradiction to the assumption. 
\end{proof}

\medskip

\begin{proof}[Proof of Corollary~\ref{cor:main}.]
Let $K$ be an L-space knot.  By mirroring, we may assume it is a positive L-space knot.  
Then by combining Hedden and Ni \cite{hedden, ni}, $K$ is a tight fibered knot.  
Krcatovich \cite{krcatovichthesis,krcatovicharxiv} shows that $K$ must be prime.   
Now Theorem~\ref{thm:main} implies $K$ cannot be a non-trivial band sum. 
Note that if $K$ is a positive L-space knot, 
then the equality $\tau(K) = g(K)$ follows by \cite{OSlens}. 
\end{proof}

\section{A family of prime tangles}
In Lemma~\ref{lem:primeK} we show Example~\ref{primeSQPband} of a band sum is prime by demonstrating that it has a prime tangle decomposition and then employing the work of Nakanishi \cite{Nakanishi}.  
The two tangles involved both have a similar form which we generalize to the family of $(n+1)$--strand tangles $T^n(\tau)$ based off a two-strand tangle $\tau$. 
(See Figure~\ref{fig:basictangles} and the discussion below.)  
In Proposition~\ref{prime tangle} we provide hypotheses that ensure the tangle $T^n(\tau)$ is prime for $n\geq1$. 

\medskip
For our purposes, an  {\em $n$--string tangle} is a proper embedding of a disjoint union of $n$ arcs into a ball considered up to proper isotopy; 
$n$ is a positive integer. 
Recall that a tangle $T$ is {\em prime} if it satisfies the following conditions: 

\begin{enumerate}
\item $T$ is not a trivial $1$--string tangle. 
\item $T$ is {\em locally trivial}: Any $2$--sphere that transversally intersects $T$ in just two points bounds a ball in which $T$ is a trivial $1$--string tangle.
\item $T$ is {\em non-splittable}:  Any properly embedded disk does not split $T$. 
\item $T$ is {\em indivisible}: Any properly embedded disk that transversally intersects $T$ in a single point divides the tangle into two subtangles, at least one of which is the trivial $1$--string tangle. 
In the following we say that such a disk is also {\em boundary-parallel}. 
\end{enumerate}

Given a two-strand tangle $\tau$ with fixed boundary so that the endpoints on the left belong to different arcs,  
define the $(n+1)$--strand tangle $T^n(\tau)$ as in Figure~\ref{fig:basictangles} for non-negative integers $n$.  
For convenience, let us take $k=T^0(\tau)=(B,k)$, $T = T^1(\tau) = (B, k \cup a)$, 
and $T' = T^2(\tau) = (B, k \cup a \cup a')$ as also shown in Figure~\ref{fig:basictangles}.  
Here $B$ is the $3$--ball.

\begin{figure}
\centering
\includegraphics[width=4in]{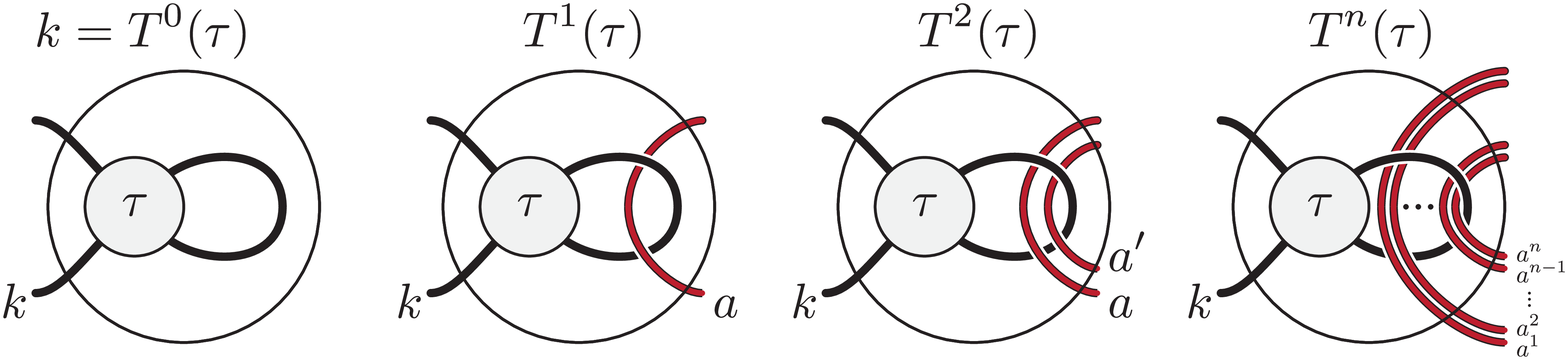}
\caption{The tangles $k=T^0(\tau)$, 
$T = T^1(\tau) = (B, k \cup a)$,  $T' = T^2(\tau) = (B, k \cup a \cup a')$, 
and $T^n(\tau)= (B, k \cup a^1 \cup a^2 \cup \dots a^{n-1} \cup a^n)$.}
\label{fig:basictangles}
\end{figure}

\begin{lemma}
\label{lem:knottedsummand}
If $T=T^1(\tau)=(B, k \cup a)$ is locally non-trivial, 
then $k$ is a knotted arc and $(B, \tau)$ is locally non-trivial, containing a summand of the knotted arc $k$. 
\end{lemma}

\begin{remark}
This summand is not necessarily proper.  
For example, the tangle $\tau$ could be homeomorphic to a split tangle consisting of $k$ and another arc. 
\end{remark}

\begin{proof}
Assume the tangle $(B, k \cup a)$  is locally non-trivial.  
Since $a$ is a trivial arc, 
any local knotting would have to occur in $k$.  
Let $S$ be a sphere bounding a ball $B_S$ that intersects $k \cup a$ in a knotted subarc of $k$. 
By replacing $S$ with a smaller one if necessary, 
we may assume the closure of the knotted $1$--string tangle $(B_S, B_S \cap k)$ is a prime knot. 
Observe that the arc $a$ together with an arc in $\bdry B$ bounds a disk $\delta$ that intersects $k$ exactly once.
We will inductively isotope $S$ and $k$ leaving $a$ invariant so that $S \cap \delta = \emptyset$.

Assume that $S \cap \delta \neq \emptyset$.   
Because $S$ is disjoint from $a$, 
$S \cap \delta$ is a collection of circles.  
A circle of $S \cap \delta$ that is innermost in $\delta$ bounds a subdisk $\delta' \subset \delta$.   

If there is such a disk $\delta'$ that is disjoint from $k$, 
then the circle $\bdry \delta'$ also bounds a subdisk $\sigma \subset S$ that is disjoint from $k$.  
Since $B-\nbhd(k)$ is irreducible (because it is the exterior of the knot $K$, 
the closure of $k$) and $\delta' \cup \sigma \subset B-\nbhd(k)$, 
we may isotope $\sigma$ to $\delta'$ (in the complement of $k \cup a$) and then further isotope $S$ to reduce $|S \cap \delta|$.  
Perform such isotopies until every circle of $S \cap \delta$ bounds a subdisk of $\delta$ that intersects $k$; 
in particular, 
the circles $S \cap \delta$ are concentric circles in $\delta$ about the single intersection point $\delta \cap k$.

Now if $S \cap \delta \neq \emptyset$, 
then the innermost one in $\delta$ bounds a disk $\delta' \subset \delta$ that meets $k$ in a single point and is contained in $B_S$.  
Then $\bdry \delta'$ divides $S$ into two disks $\sigma$ and $\sigma'$ that each intersect $k$ once and $\delta'$ divides $B_S$ into two balls bounded by spheres $\delta' \cup \sigma$ and $\delta' \cup \sigma'$.  
Since we have chosen $S$ so that $(B_S, B_S \cap k)$ is a prime knot, 
one of these balls intersects $k$ in a trivial arc.
We may use these balls to isotope $B_S$ along with $k \cap B_S$ into a collar neighborhood of $\delta'$ and then further to be disjoint from $\delta'$ (and $\delta$).  

Since $S \cap \delta = \emptyset$, $B_S$ is contained in $(B-\nbhd(\delta), k \cup a) \cong (B, \tau)$.  
Hence $\tau$ contains a locally knotted arc which is the summand of $k$ sectioned off by $B_S$.
\end{proof}

\begin{proposition}
\label{prime tangle}
If $k$ is a non-trivial arc without any proper summand in $B$, 
and $\tau$ is locally trivial, 
then $T^n(\tau)$ is prime for $n\geq 1$.
\end{proposition}

\begin{proof}
We proceed to check the four conditions of primeness for $T^n(\tau)$:

(1) This is obvious since $T^n(\tau)$ is an $n+1$--string tangle and $n \geq 1$. 

\medskip

(2) If $T^n(\tau)$ were locally non-trivial, then since the arcs $a^1, \dots, a^n$ are trivial arcs, 
the local knotting would have to occur in the component $k$.  
Thus the $2$--string tangle $T=T^1(\tau)=(B, k \cup a)$ would also be locally non-trivial (where we let $a=a^1$).  
Then Lemma~\ref{lem:knottedsummand} shows that this implies that $\tau$ would have to be a locally non-trivial $2$--string tangle containing a summand of $k$.   
This implies that $\tau$ is locally non-trivial, contradicting the assumption.

\medskip

(3)  Assume $T^n(\tau)$ is splittable.  
Then there is a disk $D$ that separates one component of $T^n(\tau)=(B, k \cup a^1 \cup \dots  \cup a^n)$ from another.  
Hence $D$ must separate $k$ from the arc $a^i$ for some $i$.  
Since the arcs $a^1, \dots, a^n$ are mutually isotopic in the complement of $k$, 
we may assume $k \cup a^1$ is splittable.   
However, this implies that $T^1(\tau) = (B, k \cup a^1)$ is locally non-trivial since $k$ is a non-trivial arc, a contradiction to (2).

\medskip

(4) Assume $T^n(\tau)$ is divisible.  
Then there is a disk $D$ transversally intersecting $T^n(\tau)$ in just one point so that $D$ is not boundary-parallel.    
Since the arcs $a^1, \dots, a^n$ are mutually isotopic, 
we may assume either $D$ intersects $k$ or $D$ intersects $a^1$.

If $D$ intersects $k$, then it must cut off a trivial arc from $k$ since $k$ has no proper summand.  
Because $D$ is not boundary-parallel, some arc of $a^1, \dots, a^n$ must be on this side, say $a^i$.  
But then $T^1(\tau)=(B,k \cup a^i)$ is locally non-trivial, a contradiction to (2).

If $D$ intersects $a^1$, then the arcs $a^2, \dots, a^n$ and $k$ must all be on the same side of $D$.  
Otherwise, or some $i \ge 2$, $T^1(\tau) = (B,k \cup a^i)$ would be a splittable $2$--string tangle contrary to (3).  
But now since $a^1$ is a trivial arc, $D$ must be boundary-parallel to the side that does not contain $a^2, \dots, a^n$ and $k$.  
Again this is a contradiction.
\end{proof}

\section{Examples}

In this section, we give examples that prove Propositions~\ref{examples} and \ref{nonSQP}. 
Examples~\ref{primeSQPband} and \ref{fibered_bund sum} give 
Proposition~\ref{examples}(1) and (2) respectively. 
Proposition~\ref{nonSQP} follows from Example~\ref{non-prime non-SQP}. 

\subsection{A prime, non-fibered, strongly quasipositive banding of tight fibered knots}

We show that there exists a banding $K=K_1 \natural_\beta K_2$ of tight fibered knots $K_1$ and $K_2$ which is prime and strongly quasipositive;  
furthermore, $K$ has the Alexander polynomial of an L-space knot. 
It follows from Theorem~\ref{thm:main} (or by Kobayashi \cite{kobayashi}) that $K$ cannot be fibered, 
and thus it is not an L-space knot \cite{ni} (cf.\ Corollary \ref{cor:main}).

\begin{lemma}
\label{block}
If $K=K_1 \natural_\beta K_2$ and $g(K)=g(K_1)+g(K_2)$ then $\Delta_K(t) = \Delta_{K_1}(t) \Delta_{K_2}(t)$.
\end{lemma}
\begin{proof}
By Gabai \cite{gabaiband}  and Scharlemann \cite{Sch}, 
when $g(K) = g(K_1)+g(K_2)$ for the band sum $K=K_1 \natural_\beta K_2$, 
then there are minimal genus Seifert surfaces $F, F_1, F_2$ for the knots $K, K_1, K_2$ such that $F$ is the union of $F_1, F_2$ and the band $\beta$.  
Since $F_1$ and $F_2$ are  separated by a sphere, 
$F$ has a Seifert matrix that is block diagonal of the Seifert matrices for $F_1$ and $F_2$.   
The result then follows.
\end{proof}

Proposition~\ref{examples}(1) follows from the example below. 

\begin{example}
\label{primeSQPband}
Let $K_1$ be the $(2, 3)$--torus knot $T_{2, 3}$ and $K_2$ the $(2, 1)$--cable of the $(2, 3)$--torus knot
$T_{2, 3}^{2, 1}$; $K_2$ wraps twice in a longitudinal direction and once in a meridional direction along the companion $T_{2, 3}$. 
Note that $K_1$ and $K_2$ are both strongly quasipositive fibered knots, and hence tight fibered knots \cite{hedden}.
The left-hand side of Figure~\ref{fig:SQPbandings} shows the split link $K_1 \sqcup K_2$.  
The right hand side of Figure~\ref{fig:SQPbandings} shows the banding $K = K_1 \natural_\beta K_2$ which, 
by virtue of its presentation is also strongly quasipositive. 

\begin{figure}
\centering
\includegraphics[width=5in]{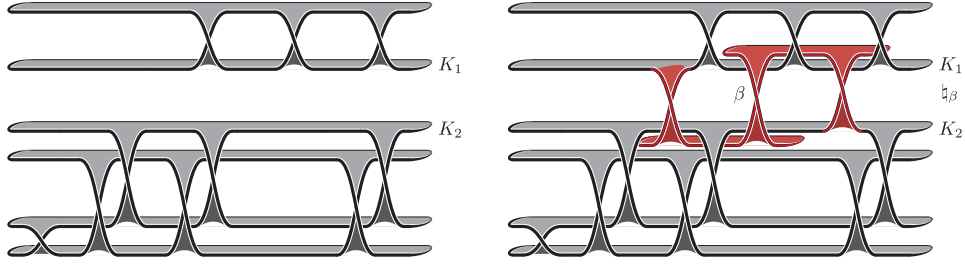}
\caption{The $(2,3)$--torus knot $K_1$ and the $(2,1)$--cable of the $(2,3)$--torus knot $K_2$ are strongly quasipositive knots.  
They may be banded together to form another strongly quasipositive knot $K=K_1 \natural_\beta K_2$.}
\label{fig:SQPbandings}
\end{figure}

By Lemma~\ref{block}, it now follows that $\Delta_K(t) = \Delta_{K_1}(t) \Delta_{K_2}(t) = (t-1+t^{-1})(t^2-1+t^{-2})$.   
Indeed, let us note that we constructed $K$ so that $\Delta_K(t)$ is also the Alexander polynomial of an L-space knot, 
namely the L-space knots $T_{3, 4}$ and $T_{2, 3}^{2, 3}$, 
i.e.\ the $(3, 4)$--torus knot and the $(2, 3)$--cable of $(2, 3)$ torus knot.   

Lemma~\ref{lem:primeK} below uses tangle theory to shows that $K = K_1 \natural_\beta K_2$ is a prime knot.
Hence by Theorem~\ref{thm:main} (or by Kobayashi \cite{kobayashi}), $K$ is not fibered.
\end{example}

\begin{lemma}
\label{lem:primeK}
The knot $K = K_1 \natural_\beta K_2$ on the right-hand side of Figure~\ref{fig:SQPbandings} is a prime knot.
\end{lemma}

\begin{proof}
Discarding the Seifert surface followed by a small isotopy presents $K = K_1 \natural_\beta K_2$ as on the left-hand side of Figure~\ref{fig:splittangles}.
The sphere separating $K_1$ and $K_2$ that intersects $K$ in $6$ points 
(shown first as a horizontal line) splits $K$ into two $3$--string tangles $T_1$ and $T_2$ as shown in Figure~\ref{fig:splittangles}.  
Then Figure~\ref{fig:twotangles} shows, for each $i=1,2$, 
the results of further isotopies expressing $T_i$ as the tangle $T^2(\tau_i) = k_i \cup a_i \cup a_i'$ for some $2$--string tangle $\tau_i$ such that $k_i$ is a knotted arc of knot type $K_i$ while $a$ and $a'$ are trivial arcs that are isotopic in the complement of $k_i$.  
(See Figure~\ref{fig:basictangles}.)  
Note that the knot $K_i$ is the closure of $\tau_i$ while the arc $k(\tau_i)$ is the ``half-closure" of $\tau_i$.  

\begin{figure}
\centering
\includegraphics[width=5in]{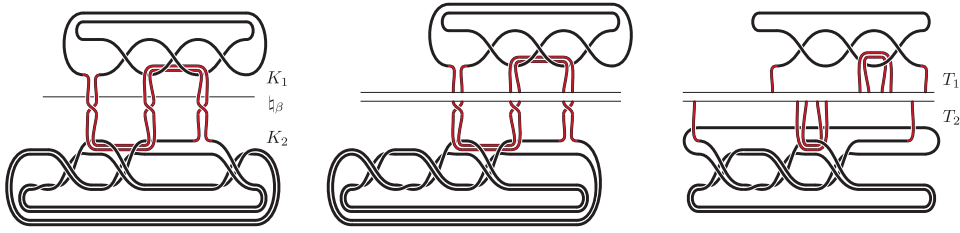}
\caption{The splitting sphere of $K_1 \sqcup K_2$ divides $K=K_1 \natural_\beta K_2$ into two three strand tangles, $T_1$ and $T_2$.  The right-hand picture of the tangles gives isotopic versions in the center.}
\label{fig:splittangles}
\end{figure}

\begin{figure}
\centering
\includegraphics[width=4in]{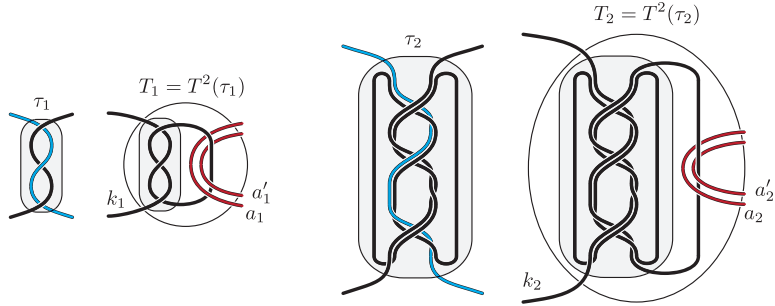}
\caption{A further isotopy of the two tangles $T_1$ and $T_2$ exhibits them as the tangles $T^2(\tau_1)$ and $T^2(\tau_2)$.}\label{fig:twotangles}
\end{figure}

Since $K_i$ is a non-trivial prime knot for each $i=1,2$, 
then the arcs $k_i$ are non-trivial without any proper summand.  
Because $\tau_1$ is a trivial tangle, it is locally trivial.  
Because $k_2$ has no proper summand, if $\tau_2$ were locally non-trivial, 
the local knotting would have the knot type of $k_2$. 
However since the two one-strand subtangles of $\tau_2$ are a trivial arc and a knotted arc which one may identify as having the knot type of $8_{20}$ which is not the type of $k_2$, 
$\tau_2$ is locally trivial.  
(Figure~\ref{fig:twotangles} highlights one strand in each $\tau_1$ and $\tau_2$ so that the other strand is more easily discerned.)
Thus Proposition~\ref{prime tangle} shows that $T_1$ and $T_2$ are both prime tangles.
Work of Nakanishi then implies that $K$ is a prime knot \cite{Nakanishi}.
\end{proof}

\medskip

\subsection{A prime, fibered,  non-strongly quasipositive banding of tight fibered knots}

We show that there exists a banding $K = K_1 \natural_\beta K_2$ of tight fibered knots $K_1$ and $K_2$ which is prime, fibered, 
but not tight, and hence not strongly quasipositive. 

\begin{example}
\label{fibered_bund sum}
Take $K_1 = T_{2, 3}$ and the trivial knot $K_2$, each of which is a tight fibered knot. 
Let $K_1 \sqcup K_2$ be a split link, and $\beta$ a band given in Figure~\ref{fig:8_10}. 
Then the band sum $K = K_1 \natural_\beta K_2$ is the prime knot $8_{10}$.   
(This band sum presentation of $8_{10}$ was given by Kobayashi \cite{kobayashi2}.)
Recall that $K$ is a Montesinos knot $M(\frac{1}{2}, \frac{1}{3}, \frac{2}{3})$ and it is easy to see that 
$K$ is a plumbing of two torus knots $T_{2, 5}$ and $T_{2, -3}$. 
Hence it is a fibered knot \cite[Theorem~1]{Stallings}. 
By Theorem~\ref{thm:main} $K$ is not tight, and hence it is not strongly quasipositive. 
\end{example}

\begin{figure}[h]
\centering
\includegraphics[width=1.8in]{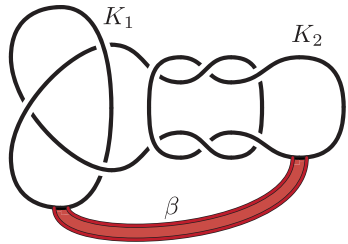}
\caption{$K_1 = T_{3, 2}$, $K_2 = \mathrm{unknot}$, and 
$K=K_1 \natural_\beta K_2$ is $8_{10}$.}
\label{fig:8_10}
\end{figure}

\subsection{A non-prime, fibered, non-strongly quasipositive banding of tight fibered knots.}
Recall that the connected sum of tight fibered knots is always tight fibered.
In this subsection we show that for any tight fibered knots $K_1$ and $K_2$, 
we can take a band $\beta$ so that $K_1 \natural_{\beta} K_2$ is fibered, but not tight, 
and hence not strongly quasipositive. 

\begin{example}
\label{non-prime non-SQP}
Figure~\ref{fig:bandsquareknotsum} shows a band $\beta$ for a split link of any two knots $K_1$ and $K_2$ 
so that the band sum $K=K_1 \natural_\beta K_2$ produces the connected sum $K = K_1\# 3_1 \# \overline{3_1} \# K_2$.   
Since the square knot $3_1 \# \overline{3_1}$ is fibered, 
if $K_1$ and $K_2$ are fibered then $K$ will be fibered.  
 However since $3_1 \# \overline{3_1}$ is a non-trivial ribbon knot, 
$\tau(3_1 \# \overline{3_1})=0$, 
hence $\tau(K_1 \# 3_1 \# \overline{3_1} \# K_2) = \tau(K_1  \# K_2) \le g(K_1 \# K_2) < g(K_1 \# 3_1 \# \overline{3_1} \# K_2)$. 
This shows $K$ cannot be strongly quasipositive, and in particular it is not tight. 
Thus, choosing $K_1$ and $K_2$ to be tight fibered knots gives the example.

\begin{figure}[h]
\centering
\includegraphics[width=2in]{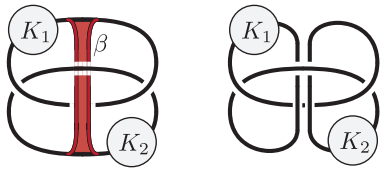}
\caption{The band $\beta$ for the split link $K_1 \sqcup K_2$  produces the band sum $K_1 \natural_\beta K_2 = K_1\# 3_1 \# \overline{3_1} \# K_2$. }
\label{fig:bandsquareknotsum}
\end{figure}
\end{example}

\section{Further discussion and questions}
Our results and  proofs lead to a few natural questions.

\subsection{Band sums of strongly quasipositive knots}

\begin{question}
If $K = K_1 \natural_\beta K_2$ is a strongly quasipositive knot, must $K_1$ and $K_2$ be strongly quasipositive?   
\end{question}

Since $\tau(K) = g(K)$ for strongly quasipositive knots, 
the argument in the proof of Theorem~\ref{thm:main} enables us 
to conclude that $g(K)=g(K_1)+g(K_2)$. 
Then the equality implies that there are minimal genus Seifert surfaces $F_1, F_2$ for the knots $K_1, K_2$ that are disjoint from the band $\beta$ so that $F = F_1 \cup \beta \cup F_2$ is a minimal genus Seifert surface for $K$; 
see Gabai \cite[Theorem~1]{gabaiband} and Scharlemann \cite[8.5 Remark]{Sch}.
If $F$ is a quasipositive Seifert surface, 
then it follows that $F_1$ and $F_2$ are also quasipositive Seifert surfaces (because they are subsurfaces of $F$) and hence $K_1$ and $K_2$ are strongly quasipositive knots.

From this point of view, an affirmative answer to this question would follow from an affirmative answer to the following. 

\begin{question}
\label{ques:mingenquasipos}
If $F$ is a minimal genus Seifert surface for a strongly quasipositive knot, must $F$ be a quasipositive surface?
\end{question}

Let us note that the Kakimizu complex for minimal genus Seifert surfaces for a knot is connected \cite{kakimizu}.  
Thus if the answer to this question is negative, then there is a strongly quasipositive knot $K$ with a quasipositive Seifert surface $Q$ and a non-quasipositive minimal genus Seifert surface $F$ such that $F \cap Q = \bdry F = \bdry Q = K$.

\subsection{Band sums of split links}
The band sum operation can be generalized as follows. 
Start with a split link with $n$ components $K_1, \dots, K_n$ 
(where for each component there is a sphere separating it from all of the other components),  
and connect the components via $n-1$ pairwise disjoint bands to obtain a knot $K$. 
We call $K$ a {\em band sum} of $K_1, \dots, K_n$; see \cite{miyazakiband}.  
As an analogy of the case $n = 2$, 
we say that a band sum is {\em trivial} if $K$ coincides with one of $K_i$ ($1 \le i \le n$). 

\begin{question}
\label{generalization}
\begin{enumerate}
\item
If a prime knot $K$ in $S^3$ is a non-trivial band sum of $K_1, \dots, K_n$, 
then can $K$ be a tight fibered knot? 
\item 
If a knot $K$ in $S^3$ is a non-trivial band sum of $K_1, \dots, K_n$, 
then can $K$ be an L-space knot? 
\end{enumerate}
\end{question}

We expect a negative answer to both of these questions.  
Since the relation $g(K_1) + \dots + g(K_n) \geq \tau(K_1) + \dots + \tau(K_n) = \tau(K)$ holds, 
if one follows the scheme of our proof of Theorem~\ref{thm:main}, 
it becomes a question of whether the ``superadditivity of genus'' for such band sums holds true and whether there is a  result similar to Kobayashi's work on fibering and band sums.

\bigskip

\begin{footnotesize}
	\bibliographystyle{plain}
	\bibliography{BakerMotegi-bandsum}

\begin{thebibliography}{10}

\bibitem{BI}
Sebastian Baader and Masaharu Ishikawa.
\newblock Legendrian graphs and quasipositive diagrams.
\newblock {\em Ann. Fac. Sci. Toulouse Math. (6)}, 18(2):285--305, 2009.

\bibitem{etnyre}
John~B. Etnyre.
\newblock Lectures on open book decompositions and contact structures.
\newblock In {\em Floer homology, gauge theory, and low-dimensional topology},
  volume~5 of {\em Clay Math. Proc.}, pages 103--141. Amer. Math. Soc.,
  Providence, RI, 2006.

\bibitem{gabaiband}
David Gabai.
\newblock Genus is superadditive under band connected sum.
\newblock {\em Topology}, 26(2):209--210, 1987.

\bibitem{geiges}
Hansj{\"o}rg Geiges.
\newblock {\em An introduction to contact topology}, volume 109 of {\em
  Cambridge Studies in Advanced Mathematics}.
\newblock Cambridge University Press, Cambridge, 2008.

\bibitem{hedden}
Matthew Hedden.
\newblock Notions of positivity and the {O}zsv\'ath-{S}zab\'o concordance
  invariant.
\newblock {\em J. Knot Theory Ramifications}, 19(5):617--629, 2010.

\bibitem{kakimizu}
Osamu Kakimizu.
\newblock Finding disjoint incompressible spanning surfaces for a link.
\newblock {\em Hiroshima Math. J.}, 22(2):225--236, 1992.

\bibitem{kobayashi}
Tsuyoshi Kobayashi.
\newblock Fibered links which are band connected sum of two links.
\newblock In {\em Knots 90 ({O}saka, 1990)}, pages 9--23. de Gruyter, Berlin,
  1992.

\bibitem{kobayashi2}
Tsuyoshi Kobayashi.
\newblock Knots which are prime on band connected sum.
\newblock In {\em Proc. Applied Mathematics Workshop ({K}aist, 1994)}, pages
  79--89. 1994.

\bibitem{krcatovicharxiv}
David Krcatovich.
\newblock The reduced knot {F}loer complex, 2013.

\bibitem{krcatovichthesis}
David~Thaddeus Krcatovich.
\newblock {\em The reduced knot {F}loer complex}.
\newblock ProQuest LLC, Ann Arbor, MI, 2014.
\newblock Thesis (Ph.D.)--Michigan State University.

\bibitem{Liv}
Charles Livingston.
\newblock Computations of the {O}zsv\'ath-{S}zab\'o knot concordance invariant.
\newblock {\em Geom. Topol.}, 8:735--742 (electronic), 2004.

\bibitem{miyazakiband}
Katura Miyazaki.
\newblock Band-sums are ribbon concordant to the connected sum.
\newblock {\em Proc. Amer. Math. Soc.}, 126(11):3401--3406, 1998.

\bibitem{Nakanishi}
Yasutaka Nakanishi.
\newblock Primeness of links.
\newblock {\em Math. Sem. Notes Kobe Univ.}, 9(2):415--440, 1981.

\bibitem{ni}
Yi~Ni.
\newblock Knot {F}loer homology detects fibred knots.
\newblock {\em Invent. Math.}, 170(3):577--608, 2007.

\bibitem{OS}
Peter Ozsv{\'a}th and Zolt{\'a}n Szab{\'o}.
\newblock Knot {F}loer homology and the four-ball genus.
\newblock {\em Geom. Topol.}, 7:615--639, 2003.

\bibitem{OSlens}
Peter Ozsv{\'a}th and Zolt{\'a}n Szab{\'o}.
\newblock On knot {F}loer homology and lens space surgeries.
\newblock {\em Topology}, 44(6):1281--1300, 2005.

\bibitem{Rudolph_Topology}
Lee Rudolph.
\newblock Constructions of quasipositive knots and links. {III}. {A}
  characterization of quasipositive {S}eifert surfaces.
\newblock {\em Topology}, 31(2):231--237, 1992.

\bibitem{Rudolph_HandBook}
Lee Rudolph.
\newblock Knot theory of complex plane curves.
\newblock In {\em Handbook of knot theory}, pages 349--427. Elsevier B. V.,
  Amsterdam, 2005.

\bibitem{Sch}
M.~Scharlemann.
\newblock Sutured manifolds and generalized {T}hurston norms.
\newblock {\em J. Diff. Geom.}, 29:557--614, 1987.

\bibitem{Stallings}
John~R. Stallings.
\newblock Constructions of fibred knots and links.
\newblock In {\em Algebraic and geometric topology ({P}roc. {S}ympos. {P}ure
  {M}ath., {S}tanford {U}niv., {S}tanford, {C}alif., 1976), {P}art 2}, Proc.
  Sympos. Pure Math., XXXII, pages 55--60. Amer. Math. Soc., Providence, R.I.,
  1978.

\end{thebibliography}
\end{footnotesize}

\end{document}